\documentclass[a4paper,11pt, twoside]{article} 

\usepackage{amsfonts,amssymb,amsthm,amsmath}
 \usepackage[dvips]{graphics}
\usepackage{subfigure}
 \newcounter{idesc} 

\DeclareSymbolFont{AMSb}{U}{msb}{m}{n}

 \def\P2m{\P^2(M)}

\def\N{\mathbb N}

\def\R{\mathbb R}

\def\mint{\,\mathop{\makebox[0pt][c]{$\int$}\makebox[0pt][c]{--}\,}}

\def\E{\mathcal E}

\def\1{\,{\makebox[0pt][c]{\normalfont    1}
\makebox[2.5pt][c]{\raisebox{3.5pt}{\tiny {$\|$}}}
\makebox[-2.5pt][c]{\raisebox{1.7pt}{\tiny {$\|$}}}
\makebox[2.5pt][c]{} }}

\newcommand{\norm}[1]{\left\| #1 \right\|}
\def\P{{ /\!\!/}}

\def\for{\mbox{ for }}

\def\ol{\overline}

\def\supp{\mathop{\mbox{supp}}}

\def\divergence{{{\mathop{\,{\rm div}}}}}

\def\dist{{\mbox{dist}}}

\def\Lip{\mathop{\mbox{\normalfont Lip}}}

\def\dist{\mbox{\normalfont dist}}

\newcommand{\D}{{\mathbb{D}}}

\theoremstyle{definition}

\def\P{{\mathcal P}}

\def\L2{{L^2}}
\def\L11{\mathcal P _{ac}}
\def\S1{{S^1}}

\newcounter{zaehler}
\def\bbox{{\hfill $\Box$}}

\def\id{{\mbox{id}}}
\def\L{{\mathcal L}}

\def\D{{\mathcal D}}

\newtheorem{theorem}{Theorem}[section]
\newtheorem{lemma}[theorem]{Lemma}

\newtheorem{corollary}[theorem]{Corollary}
\newtheorem{proposition}[theorem]{Proposition}
\newtheorem{remark}[theorem]{Remark}

\def\indicator{{\mathchoice {1\mskip-4mu\mathrm l}%
{1\mskip-4mu\mathrm l}{1\mskip-4.5mu\mathrm l}%
{1\mskip-5mu\mathrm l}}}

\DeclareMathOperator{\capacity}{cap}

\title{Regularity Properties for a System of Interacting Bessel Processes}
\author{ Sebastian Andres and Max-K. von Renesse\\
  \footnote{\today, Technische Universit\"at Berlin,
 email: [andres,mrenesse]\@@math.tu-berlin.de,
\textbf{Keywords:} Bessel process, Coulomb Interaction, Reflecting
Boundary Condition, Feller Property, Muckenhoupt Weights} }
\date{}

\numberwithin{equation}{section}
 
\begin{document}

 \maketitle
\vspace{-4em}

\begin{center}
{\bf Abstract}\\

\bigskip

\begin{minipage}{14.3cm}
{We study the regularity of a diffusion on a simplex with singular
drift and reflecting boundary condition which describes  a finite
system of particles on an interval with Coulomb interaction and
reflection between nearest neighbors.

\smallskip As our main result we establish the Feller property
for the process in both cases of repulsion and attraction. In
particular the system can be started from any initial state,
including multiple point configurations. Moreover we show that the
process is a Euclidean semi-martingale if and only if the
interaction is repulsive. Hence, contrary to classical results about
reflecting Brownian motion in smooth domains, in the attractive
regime a construction via a system of Skorokhod SDEs is impossible.
Finally, we establish exponential heat kernel gradient estimates in
the repulsive regime.

\smallskip The main proof for the attractive case is based on  potential theory in Sobolev spaces with Muckenhoupt
weights. }
\end{minipage}

\end{center}



\section{Introduction and Main Results}
The study of Brownian motion or more general diffusions in Euclidean
domains with reflecting boundary condition is a classical subject in
stochastic analysis, with strong connection to boundary regularity
theory for parabolic PDE. Starting from the early works by e.g.\
Fukushima \cite{MR0231444} and Tanaka \cite{MR529332},
the field has seen perpetual research activity, c.f.
\cite{MR1106272,MR745330} (and e.g.\ \cite{MR2124641} for a more
comprehensive list of references). Typically the reflecting process
can be obtained in two ways. Either by solving a system of Skorokhod
SDE involving the local time at the boundary or via Dirichlet form
methods, and under suitable smoothness assumptions on the domain and
the coefficients both approaches are equivalent.

\smallskip In this paper we study a very specific singular case, in
which the drift coefficients of the operator may diverge and where
the equivalence of the two approaches breaks down but the process
exhibits good spatial regularity nevertheless. Our case corresponds
to a Dirichlet form obtained from the closure of the quadratic form
\[ \E(f,f) = \int_{\Omega} |\nabla f|^2 q(dx), \quad f\in C^{1}(\overline \Omega)\]
on $L^2(\Omega, dq)$, for a very specific choice of domain
$\Omega\subset \R^N$ and measure $q(dx)$ (see below). Similar
variants were studied 
under the smoothness assumption $q \in H^{1,1}(\Omega)$  and $\nabla
(\log q) \in L^{p}(\Omega,dq)$,
 $p>N$
in \cite{MR2021192}  and  \cite{MR2321042} respectively, where a
Skorokhod decomposition for the induced process still holds.

The case treated in this paper is obtained by choosing
\[\Omega = \Sigma_N = \{
x\in \R^N: \,  0 < x^1 < x^2 \cdots < x^N < 1  \}\] and
\begin{align} \label{def_qn}
 q_N (dx) =\frac{1}{Z_\beta} \prod_{i=0}^{N} (x^{i+1}-x^{i})^{\frac \beta {N+1} -1}dx^1\, dx^2\cdots\,
dx^N,
\end{align}
where by convention $x^0=1$, $x^{N+1}=1$,  $Z_\beta=(
{\Gamma(\beta)}/{(\Gamma(\beta/(N+1)))^{N+1}})^{-1}$ is a
normalization constant and $\beta>0$ is a free parameter.

\smallskip  We study the process $(X^N_\cdot)$ generated by (the $L^2(\Sigma_N,
q_N)$-closure of) the (pre-)Dirichlet form
\begin{equation*}
 \E^N(f,f) = \int_{\Sigma_N} | \nabla f|^2(x) \, q_N (dx) , \quad f \in C^\infty
 (\overline{\Sigma}_N),
\label{thedirichletform}
\end{equation*}
whose generator $\mathcal L$ extends the operator $(L^N, D(L^N ))$
\begin{equation}
L^N f(x)= (\frac \beta {N+1} -1 )   \sum_{i=1}^{N}\left ( \frac 1
{x^i - x^{i-1}} - \frac 1 {x^{i+1} - x^{i}} \right) \frac {\partial
}{\partial x^i} f(x) + \Delta f(x)\quad \for x\in \Sigma_N
\label{generator}
\end{equation}
 with domain
\[
D(L^N)= C^2_{Neu}= \{ f \in C^2(\overline{\Sigma}_N)\, |\, \nabla f
\cdot \nu =0 \mbox{ on all $(n-1)$-dimensional faces of } \partial
\Sigma_N\},\] and $\nu$ denoting the outward normal field on
$\partial \Sigma_N$.

\smallskip On the level of formal It\^{o} calculus $(L^N,D(L^N))$
corresponds to an order preserving dynamics for the location of $N$
particles in the unit interval which solves the system of coupled
Skorokhod SDEs
\begin{align}
d x^i _t = (\frac \beta {N+1} -1 ) \left ( \frac 1 {x^i_t -
x^{i-1}_t} - \frac 1 {x^{i+1}_t - x^{i}_t} \right) dt +\sqrt 2
dw^i_t + dl_t^{i-1} -dl_t^i, \quad i= 1, \cdots, N\label{bessde1}
\end{align}
where  $\{w_i\}$ are independent real Brownian motions and $\{l^i\}$
are the collision local times, i.e.\ satisfying
\begin{equation}
  dl_t^i \geq 0, \quad l_t^i = \int _0^t \indicator_{\{x^i_s=x_s^{i+1}\}} dl_s^i.
\label{bessde2}
\end{equation}
 $(X^N_\cdot)$ may thus be considered as a system of coupled two sided real Bessel processes with uniform Bessel dimension
 $\delta=   \beta /({N+1})$. Similar to the standard real Bessel process $BES(\delta)$ with Bessel dimension $\delta < 1$,
 the existence of $X^N$, even with initial condition initial condition $X_0 = x\in
 \Sigma_N$, is not trivial, nor are its regularity properties.\\

Our  motivation for studying this process is its relation to the
Wasserstein diffusion, c.f.\ \cite{vrs07}. In \cite{pws} we showed
that the normalized empirical measure of the system \eqref{bessde1}
converges to the Wasserstein diffusion in the high density regime
for $N\to\infty$. Hence the regularity properties of $(X^N_\cdot)$
may give an indication of the regularity of the Wasserstein
diffusion, although in the present paper we  treat the  case when
the dimension $N$ is fixed. Here our results read as follows.
\begin{theorem} \label{fellerthm} {\it For any $ \beta >0$, the Dirichlet form $\E^N$ generates a Feller process
$(X_\cdot)$ on $\overline{\Sigma}_N$, i.e. the associated transition
semigroup on $L^2(\Sigma_N, q_N)$ defines a strongly continuous
contraction semigroup on the subspace $C(\overline{\Sigma}_N)$
equipped with the sup-norm topology.\\ Moreover, for $\beta \geq
(N+1)$ the  associated heat kernel $(P_t)$ is exponentially
smoothing on Lipschitz functions, i.e. for $t>0$
\begin{equation}
\Lip(P_t f) \leq\exp\bigl(-(\frac \beta {N+1} -1) k_N \cdot t\bigr)
\Lip(f)  \label{contraction}
\end{equation}
for all $f\in $Lip$(\overline\Sigma_N)$, where $k_N>0$ is a
universal constant depending only on $N$.\ }
\end{theorem}

\smallskip  In the proof of Theorem \ref{fellerthm} for the more difficult case $\beta < (N+1)$  we use some
localization arguments to exploit the geometric symmetries of the
problem. A crucial ingredient for this approach is the following
('Markov-')uniqueness result for the operator
 $(L^N,C^2_{Neu})$ which is interesting in its own right.

\begin{proposition} \label{uniqueness} \textit{For $\beta<2(N+1)$, there is at most one
symmetric strongly continuous contraction semigroup $(T_t)_{t\geq
0}$ on $L^2(\Sigma_N, q_N)$ whose generator  $(\mathcal L , \mathcal
D(\mathcal L )) $ extends $(L^N,C^2_{Neu})$.}
\end{proposition}
\index{\footnote{}}
Hence, together with Theorem \ref{fellerthm} we can state the
following existence and uniqueness result.

\begin{corollary}\textit{
The formal system of Skorokhod SDEs defines via the associated
martingale problem a unique diffusion process which can be started
everywhere on the (closed) simplex $\overline{\Sigma}_N$.}
\end{corollary}

As for path regularity we obtain the following characterization.

\begin{theorem} \label{thmsemimart} \textit{For any starting point $x\in \overline{\Sigma}_N$, $(X^x_\cdot)\in \overline{\Sigma}_N\subset \R^N$ is a Euclidean
semi-martingale if and only if $\beta/(N+1)\geq 1$.}
\end{theorem}
In particular we obtain that a Skorokhod decomposition of the
process $X^N_\cdot$ is impossible if $\beta$ is small enough. This
is in sharp contrast to all aforementioned previous works. Moreover,
again due to the uniqueness assertion of Proposition
\ref{uniqueness} the following negative result holds true.

\begin{corollary} \textit{If $\beta/(N+1)< 1$, the system of equations \eqref{bessde1},\eqref{bessde2} is ill-posed, i.e.\ it admits no
solution in the the sense of It\^{o} calculus.}
\end{corollary}

Theorems \ref{fellerthm} and \ref{thmsemimart} generalize the
corresponding classical results for the familiy of standard real
Bessel processes, which are proved in a completely different manner,
c.f. \cite{MR1030728,MR1725357}.\\

\textit{Strategy of proof.} While the proof of Theorem
\ref{thmsemimart} consists of a straightforward application of a
regularity criterion for Dirichlet processes by Fukushima
 \cite{MR1741537}, the proof of  Theorem \ref{fellerthm} is
more involved, and entirely different methods are used in the two
respective cases $\beta < N+1$ and $\beta \geq N+1$. Both cases are
non-trivial from an analytic point of view due to the degenerating
coefficients and the Neumann boundary condition.\\ By comparison the
case $\beta \geq (N+1)$ is much easier since $q_N$ is then
log-concave. Using a recent powerful result by Ambrosio, Savar\'e
and Zambotti \cite{asz} on the stability of gradient flows for the
relative entropy functional on Wasserstein the contraction estimate
\eqref{contraction} is established by smooth approximation and coupling.\\
 For
$\beta < (N+1)$ the measure $q_N$ is no longer log-concave and to
prove the Feller property we proceed in four main steps. The first
crucial observation is that the reference measure $q_N$ can locally
be extended to a measure $\hat q$ on the full Euclidean space which
lies in the Muckenhoupt class $\mathcal A_2$, allowing for a rich
potential theory. In particular the Poincar\'e inequality and
doubling condition hold which imply via heat kernel estimates the
regularity of the induced process on $\R^N$.  The second step is the
probabilistic piece of the localization, which corresponds to
stopping the $\R^N$-valued process at domain boundaries where the
measure $\hat q$ is 'tame' and to show that the stopped process is
again Feller. To this end we employ a version of the Wiener test for
degenerate elliptic diffusions by Fabes, Jerrison and Kenig
\cite{MR688024}. The third step is to use the reflection symmetry of
the problem which allows to treat the Neumann boundary condition
indeed via a reflection of the extended $\R^N$-valued process. The
fourth step is to establish the Markov uniqueness which is crucial
in order to justify the identification of the processes after
localization, and here we shall again depend on the nice potential
theory available in the Muckenhoupt class.

\smallskip The partial resemblance of our proof for $\beta < (N+1)$ to the
classical work by Bass and Hsu \cite{MR1106272} is no surprise.
However, we did not find any similar work in the probability
literature where potential theory for Muckenhoupt weights was used
so extensively as in our case. We think that the general strategy
outlined here may be useful also in other, less specific situations.

\section{Dirichlet Form and Integration by Parts Formula}
We start with the rigorous construction of $(X^N_\cdot)$, which departs from the symmetrizing
 measure $q_n$ on ${\Sigma}_N$ defined above in \eqref{def_qn}.
Note that one can identify $L^p(\Sigma_N,q_N)$ with
$L^p(\overline{\Sigma}_N,q_N)$, $p\geq 1$. Throughout the paper
$q_N$ denotes both the measure and its Lebesque density, i.e.\
$q_N(A)=\int_A q_N(x) \, dx$ for all measurable $A\subseteq
\overline{\Sigma}_N$.

 For all $\beta >0$, $N\in \N$, the measure
$q_N$  satisfies the 'Hamza condition', because it has a strictly
positive density with locally integrable inverse, c.f.\ e.g.\
\cite{MR1038449}.\nocite{MR2021192} This implies that the form $
\E^N(f,f)$ with domain $f  \in C^\infty (\overline{\Sigma}_N)$ is
closable on $L^2(\Sigma_N, q_N)$. The $L^2(\Sigma_N, q_N)$-closure
defines a local regular Dirichlet form, still denoted by $\E^N$.
General Dirichlet form theory asserts the existence of a Hunt
diffusion  $(X^N_\cdot)$ associated with $\E^N$ which can be started
in $q_N$-almost all $x\in \overline{\Sigma}_N$ and which is
understood as a generalized solution of the system
\eqref{bessde1},\eqref{bessde2}. This identification is justified by
the fact that any semi-martingale solution solves via It\^{o}'s
formula the martingale problem for the operator $(L^N,D(L^N ))$,
defined in \eqref{generator}.\\

The following integration by parts formula for $q_N$ can be easily
verified (e.g.\ by approximation from corresponding integrals over
increasing sub-domains $\Sigma^\epsilon_N \subset \Sigma_N$).

 \begin{proposition} \label{ibpf}
\textit{
Let $u\in C^1(\overline{\Sigma}_N)$ and $\xi=(\xi^1,\ldots,\xi^N)$ be a vector field in $C^1(\overline{\Sigma}_N, \mathbb{R}^N)$ satisfying  $\langle \xi, \nu \rangle =0$ on $\partial \Sigma_N$, where $\nu$ denotes the outward normal field of $\partial \Sigma_N$.
 Then,
\[
 \int_{\Sigma_N} \langle \nabla u, \xi \rangle q_N(dx)=-\int_{\Sigma_N} u \left[  \divergence(\xi)+ (\tfrac{\beta}{N+1}-1) \sum_{i=1}^{N} \xi^i \left( \frac 1 {x^i - x^{i-1}} - \frac 1 {x^{i+1} - x^{i}} \right)   \right] q_N(dx).
\]
}
\end{proposition}

\begin{remark} \label{rem_ibpf}
 Let $u\in C^1(\overline{\Sigma}_N)$ and $\xi$ be a vector field of the the form $\xi=w \vec\varphi$ with $w\in C^1(\overline{\Sigma}_N)$ and $\vec\varphi(x)=(\varphi(x^1),\ldots, \varphi(x^{N}))$,
$\varphi \in  C^\infty([0,1])$ and $\langle \vec \varphi, \nu
\rangle =0$ on $\partial \Sigma_N$, in particular
$\varphi_{|\partial[0,1]}=0$.  Then, the integration by parts
formula above 
reads
\[
 \int_{\Sigma_N} \langle \nabla u, \xi \rangle q_N(dx)=-\int_{\Sigma_N} u \left[ w V^\beta_{N,\varphi}+\langle \nabla w, \vec \varphi\rangle \right] q_N(dx),
\]
where
\[V_{N,\varphi}^\beta(x^1,\ldots,x^{N}):= (\frac{\beta}{N+1}-1) \sum_{i=0}^{N} \frac{\varphi(x^{i+1})-\varphi(x^i)}{x^{i+1}-x^i}+\sum_{i=1}^{N}\varphi'(x^i).
\]
\end{remark}

Let $C^2_{Neu}=\{ f\in C^2(\overline{\Sigma}_N):\, \langle \nabla f,\nu\rangle=0 \mbox{ on $\partial\Sigma_N$} \}$ as above with $\nu$ still denoting the outer normal field on $\partial \Sigma_N$. Then, for any $f\in C^1(\overline{\Sigma}_N)$ and $g\in C^2_{Neu}$ we apply the integration by parts formula in Proposition~\ref{ibpf} for $\xi=\nabla g$ to obtain
\[
  \E^N(f,g)=-\int_{\Sigma_N} f \, L^N g \, q_N(dx).
\]
Moreover, it is easy to show that
\[
 \left| \E^N(f,g) \right|\leq C \, \| f\|_{L^2(\Sigma_N,q_N)}, \qquad \forall f\in D(\E^N).\]

In particular, $C^2_{Neu}$ is contained in the domain of the generator $\mathcal{L}$ associated with $\E^N$ and $L^Nf=\mathcal{L} f$ for all $f\in C^2_{Neu}$.

\section{Feller Property}
This section is devoted to the proof of Theorem \ref{fellerthm},
where we treat the cases $\beta < (N+1)$ and $\beta \geq (N+1)$
separately. Both cases are not trivial from an analytic perspective
due to the combination of degenerate coefficients and the Neumann
boundary condition. In fact we could not find any general result in
the PDE literature which contains the current model. Our proof for
the case $\beta < (N+1)$ avoids an explicit treatment of the Neumann
boundary condition via a reflection argument. In the case $\beta
\geq  (N+1)$ we use a powerful stability property of log-concave
measures.

\subsection{Case $\beta <  (N+1)$.}
For $\beta < (N+1)$ the measure $q_N$ is no longer log-concave.
However, the proof below extends also to all cases when $\beta <
2(N+1)$. Let \[ \Omega_N:=\Omega^\delta_N:=\overline{\Sigma}_N \cap
\{ x\in \R^{N}: x_{N}\leq 1-\delta \}, \] for some positive small
$\delta <[2(N+2)]^{-1}$ and the weight function
\begin{align} \label{def_hatqn}
 \hat q(x):= \hat q_{N,\delta}(x):=
\begin{cases}
q_N(x) & \text{if $x\in \Omega_N$}, \\
\frac{1}{Z_\beta} {\delta}^{\frac \beta {N+1}-1} \prod_{i=1}^{N}
(x^i-x^{i-1})^{\frac \beta {N+1} -1} & \text{if $x\in S_1 \backslash
\Omega_N$},
\end{cases}
\end{align}
where
\[
 S_1:=\{x\in \R^{N} :\, 0\leq x^1 \leq x^2 \leq \ldots \leq x^{N} \}.
\]
We want to extend the weight function $\hat q$ to the whole $\R^{N}$. To do this we introduce the mapping
\begin{align} \label{def_T}
 T: \R^{N} \rightarrow S_1 \quad  x\mapsto (|x^{(1)}|,\ldots, |x^{(N)}|),
\end{align}
where $(1).\ldots,(N)$ denotes the permutation of $1,\ldots,N$ such that
\[
 |x^{(1)}|\leq\ldots \leq |x^{(N)}|.
\]
The extension of $\hat q$ on $\R^{N}$ is now defined via $\hat q(x)=\hat q(Tx)$, $x\in \R^{N}$. Again we will also denote by $\hat q$ the induced measure on $\R^N$. Consider
the $L^2(\R^{N}, \hat q)$-closure of
\begin{equation}
 \hat\E^{N,a}(f,f) = \int_{\R^{N}}  \langle a \nabla f, a \nabla f(x) \rangle \, \hat q (dx) , \quad f \in C^\infty_c (\R^{N})
\label{globaldirichletform}
\end{equation}
still denoted by $\hat\E^{N,a}$  for a measurable field  $x\mapsto
a(x)\in \R^{N\times N}$ on $\mathbb{R}^N$ satisfying
\begin{align} \label{ass_a}
 \frac 1 c \cdot E_N\leq a(x)^t \cdot a(x) \leq c \cdot E_N
\end{align}
in the sense of non-negative definite matrices.
 Let $(Y_t)_{t\geq 0}=(Y^{N,a}_t)_{t\geq 0}$ be the associated symmetric Hunt process on $\R^{N}$, starting from the invariant distribution $\hat q$.
Finally, we denote by $(Q_t)_t$ the transition semigroup of $Y$.

\subsubsection{Feller Properties of $Y$} Let  $C_0(\mathbb{R}^N)$ be the space of continuous functions on $\R^N$
vanishing at infinity.
\begin{proposition} \label{Yfeller}
\textit{ For $2(N+1)>\beta$ and a matrix $a$ satisfying
\eqref{ass_a}, $Y^{N,a}$ is a Feller process o $\R^N$, i.e.\
\begin{enumerate}
\item[i)] for every $t>0$ and every $f\in C_0(\R^N)$ we have $Q_t f \in C_0(\R^N)$,
\item[ii)] for every $f\in C_0(\R^N)$, $\lim_{t\downarrow 0} Q_t f=f$ pointwise in $\R^N$.
\end{enumerate}
Moreover, $Q_t f\in C(\R^N)$  for every $t>0$ and every $f\in L^2(\R^N,\hat q)$.
}
\end{proposition}

\begin{remark} \label{rem_Y}
It is well known that i) and ii) even imply that $\lim_{t\downarrow 0} \| Q_t f - f \|_\infty=0$ for each $f\in C_0(\mathbb{R}^{N})$.
Moreover, the following version of the strong Markov property holds. 
Let $T$ be a stopping time with $T\leq t_0$ a.s.\ for some $t_0>0$.
Then,  for each  $f\in L^2(\mathbb{R}^{N},\hat q)$
\begin{align*}
E\left[ f(Y_{t_0}) \left| \mathcal{F}_T  \right. \right]=E_{Y_T}\left[ f(Y_{t_0-T}) \right],
\end{align*}
with $(\mathcal{F}_t)_{t\geq 0}$ denoting the natural filtration of $Y$.
\end{remark}

\begin{proof}[Proof of Proposition~\ref{Yfeller}]
ii) follows directly by path-continuity and dominated convergence. i) as well as the additional statement follow from the analytic regularity theory of symmetric diffusions, see \cite{MR1387522}, in particular Theorem 3.5 and Proposition 3.1 and Corollary 4.2,
provided the following two conditions are fulfilled:
\begin{itemize}
 \item The measure $\hat q$ is doubling, i.e.\  there exists a constant $C'$, such that for all Euclidean balls $B_R\subset B_{2R}$
\[
 \hat q(B_{2R})\leq C' \hat q(B_R).
\]
\item $\hat \E$ satisfies a uniform local Poincar\'e inequality, i.e.\ there is a constant $C'>0$ such that
\[
 \int_{B_R} |f-(f)_{B_R}|^2 d\hat q \leq C' R^2 \int_{B_R} |\nabla f|^2 d\hat q,
\]
for all Euclidean balls $B_R$ and $f\in \D(\hat \E)$, where
$(f)_{B_R}$ denotes the integral $\frac 1 {\hat q(B_R)}\int_{B_R} f
d\hat q$.
\end{itemize}
Both conditions are verified once we have proven that the weight function $\hat q$ is contained in the Muckenhoupt class $\mathcal{A}_2$, which will be done in Lemma~\ref{muckenhoupt} below. Indeed, the doubling property follows immediately from the Muckenhoupt condition (see e.g.\ \cite{MR1774162} or \cite{MR869816}) and for the proof of the Poincar\'e inequality see Theorem 1.5 in \cite{MR643158}.
\end{proof}

\begin{lemma}
\label{muckenhoupt} \textit{For $N$ such that $ \beta < 2(N+1)$, we
have $\hat q \in \mathcal{A}_2$, i.e.\ there exists a positive
constant $C=C(N,\delta)$ such that for every Euclidean ball $B_R$
\[
\frac{1}{|B_R|} \int_{B_R} \hat q \,  dx  \, \frac{1}{|B_R|} \int_{B_R} \hat q^{-1} \,  dx\leq C,
\]
where $|B_R|$ denotes the Lebesgue measure of the ball $B_R$.
}
\end{lemma}
It suffices to prove the Muckenhoupt condition for the weight function $\tilde q$ defined by
\begin{align} \label{def_qtilde}
 \tilde q(x):=\prod_{i=1}^{N} (x^i-x^{i-1})^{\beta/(N+1)-1}, \quad x\in S_1,
\end{align}
and $\tilde q(x):=\tilde q(Tx)$ if $x\in \R^{N}$, since there exist positive constants $C_1$ and $C_2$ depending on $\delta$ and $N$ such that
\begin{align}\label{hat_equiv_tilde}
 C_1 \, \tilde q(x) \leq \hat q(x) \leq C_2 \, \tilde q(x), \qquad \forall x\in \R^{N}.
\end{align}
Note that  $(1-x_{N})^{\frac \beta {N+1}-1}$ is uniformly bounded and bounded away from zero on $\Omega_N$.

Below $P_R(m)$ denotes the parallelepiped in $\R^{N}$ with basis
point $m=(m_1,\ldots,m_{N})$, which is spanned by the vectors
$v_i=\sum_{j=i}^{N} e_j$, $i=1,\ldots, N$, normalized to the length
$R$, where $(e_j)_{j=1,\ldots, N}$ is the canonical basis in
$\R^{N}$. Then, $P_R(m)$ can also be written as
\[
P_R(m)=m+\left\{ x\in \R^{N}: x^1\in \left[0,\tfrac{R}{\sqrt{N}}\right],x^2\in \left[x^1,x^1+\tfrac{R}{\sqrt{N-1}}\right], \ldots, x^{N}\in \left[x^{N-1},x^{N-1}+R\right] \right\}.
\]


\begin{lemma}
 \label{lem_parall}
\textit{ Let $B_R$ be an arbitrary Euclidean ball in $\R^{N}$. Then,
there exists a positive constant $C$ only depending on $N$ and a
parallelepiped $P_{l R}(m) \subset S_1$ with  $l>0$ independent of
$B_R$ such that
\[
 \int_{B_R} \tilde q^{\pm 1} dx \leq C  \int_{P_{l R}(m)} \tilde q^{\pm 1} dx .
\]
}
\end{lemma}

\begin{proof}
We denote by $(S_n)$, $n=1,\ldots, 2^{N}\cdot N!$, the subsets of $\R^{N}$ taking the form
\[
 S_n=\{ x\in \R^{N}: \, (s_1 \pi_1(x),\ldots, s_{N} \pi_{N}(x)) \in S_1\},
\]
where $s_i \in \{-1,1\}$, $i=1,\ldots,N$, and $\pi(x)=( \pi_1(x),\ldots,  \pi_{N}(x))$ is a permutation of the components of $x$. Note that $\R^{N}=\bigcup_n S_n$ and the intersections of the sets $S_n$ have zero Lebesgue measure.

Let now $B_R(M)$ be an arbitrary Euclidean ball with radius $R$
centered in $M\in \R^{N}$. We first consider the case $|M|\leq 2R$.
Then, obviously $B_R(M)\subset B_{4R}(0)$ and $B_{4R}(0) \cap
S_1\subset P_{4R}(0)$. Hence,
\begin{align*}
 \int_{B_R(M)} \tilde q^{\pm 1} dx \leq  \int_{B_{4R}(0)} \tilde q^{\pm 1} dx= \sum_n \int_{B_{4R}(0)\cap S_n} \tilde q^{\pm 1} dx.
\end{align*}
By the definition of $\tilde q$ we have
\[
 \int_{B_{4R}(0)\cap S_n} \tilde q^{\pm 1} dx=\int_{B_{4R}(0)\cap S_1} \tilde q^{\pm 1} dx
\]
for all $n\in \{1,\ldots, 2^{N} \cdot N!\}$. Thus,
\begin{align*}
 \int_{B_R(M)} \tilde q^{\pm 1} dx \leq 2^{N} \cdot N!  \int_{B_{4R}(0)\cap S_1} \tilde q^{\pm 1} dx \leq  C  \int_{P_{4R}(0)} \tilde q^{\pm 1} dx .
\end{align*}
Suppose now that $|M|>2R$. Let $n_0$ be such that $M\in S_{n_0}$. Then, by construction of $\tilde q$ we have
\[
 \int_{B_{R}(M)\cap S_n} \tilde q^{\pm 1} dx \leq \int_{B_{R}(M)\cap S_{n_0}} \tilde q^{\pm 1} dx
\]
for all $n=1,\ldots, 2^{N}\cdot N!$. Hence,
\begin{align*}
 \int_{B_R(M)} \tilde q^{\pm 1} dx = \sum_n \int_{B_{R}(M)\cap S_n} \tilde q^{\pm 1} dx \leq 2^{N} \cdot  N!  \int_{B_R(M)\cap S_{n_0}} \tilde q^{\pm 1} dx .
\end{align*}
Set $K:=T(B_R(M) \cap S_{n_0}) \subset S_1$ with $T$ defined as above. Then, it is clear by definition of $\tilde q$ that
\[
 \int_{B_R(M)\cap S_{n_0}} \tilde q^{\pm 1} dx = \int_{K} \tilde q^{\pm 1} dx
\]
and we get
\[
  \int_{B_R(M)} \tilde q^{\pm 1} dx \leq C  \int_{K} \tilde q^{\pm 1} dx .
\]
Finally, we choose a parallelepiped $P_{2R}(m)$ such that $K\subseteq P_{2R}(m) \subset S_1$, which completes the proof.
\end{proof}

\begin{proof}[Proof of Lemma \ref{muckenhoupt}]
We prove the Muckenhoupt condition for $\tilde q$. Recall that we
have assumed $\beta<2(N+1)$. In the following the symbol $C$ denotes
a positive constant depending on $N$ and $\delta$ with possibly
changing its value from one occurence to another. Using Lemma
\ref{lem_parall} we have
\begin{align*}
\frac{1}{|B_R|^2} \int_{B_R} \tilde q \,  dx  \,  \int_{B_R} \tilde q^{-1} dx
\leq & C R^{-2N} \int_{P_{l R}(m)} \tilde q \,  dx  \,  \int_{P_{l R}(m)} \tilde q^{-1} dx \\
=& C R^{-2N} \int_{P_{l R}(m)} \prod_{i=1}^{N} (x^i-x^{i-1})^{\beta/(N+1)-1} \,  dx  \\
& \times  \int_{P_{l R}(m)} \prod_{i=1}^{N} (x^i-x^{i-1})^{-(\beta/(N+1)-1)} dx.
\end{align*}
By the change of variables $y_i=x^i-x^{i-1}$, $i=1,\ldots, N$, we obtain
\begin{align*}
\frac{1}{|B_R|^2} \int_{B_R} \tilde q \,  dx  \,  \int_{B_R} \tilde q^{-1} dx
\leq  & C R^{-2N} \prod_{i=1}^{N} \int_{\tilde m_i}^{\tilde n_i} y_i^{\frac{\beta}{N+1}-1} dy_i \int_{\tilde m_i}^{\tilde n_i} y_i^{-(\frac{\beta}{N+1}-1)} dy_i,
\end{align*}
where we have set $\tilde m_i:=m_i-m_{i-1}$ with $m_0:=0$ and $\tilde n_i:=\tilde m_i+\tfrac{l R}{\sqrt{N+1-i}}$ for abbreviation.
Recall that in one dimension the weight function $x\mapsto|x|^\eta$ on $\R$ is contained in $\mathcal{A}_2$ if $\eta \in (-1,1)$ (see p.\ 229 and p.\ 236 in \cite{MR869816}). Hence, we get for every $i\in \{1,\ldots, N\}$
\begin{align*}
 \int_{\tilde m_i}^{\tilde n_i} y_i^{\frac{\beta}{N+1}-1} dy_i \, \int_{\tilde m_i}^{\tilde n_i} y_i^{-(\frac{\beta}{N+1}-1)} dy_i\leq C \left| \frac{lR}{\sqrt{N+1-i}} \right|^2,
\end{align*}
and the result follows.

\end{proof}

\subsubsection{Feller Properties of $Y$ inside a Box} Let
$E:=E^\delta:=\{ x\in \mathbb{R}^{N}: \, \| x\|_\infty <1-2\delta
\}$ be a box in $\mathbb{R}^{N}$ centered in the origin and
$\mathcal{B}(E)$ be the Borel $\sigma$-field on $E$. We denote by
\[ \
\tau_E:=\inf\{ t>0: \, Y_t \in E^c \}
\]
 the first exit time of $E$. This subsection is devoted to the proof of the Feller properties for the stopped process $Y^{E}_t:=Y^{N,a,E}_t:=Y^{N,a}_{t\wedge \tau_E}$, whose transition semigroup is given by
\[
 Q_t^E f(x)=E^x[f(Y_{t\wedge \tau_E})]=E^x\left[f(Y_{t}) \indicator_{\{t<\tau_E\}} +f(Y_{\tau_E}) \indicator_{\{t\geq\tau_E\}}  \right], \qquad t>0, \, x\in \bar E,
\]
for every bounded $f$ on $\bar E$. In order to prove the Feller property
we will follow essentially the proof of Theorem~13.5 in \cite{MR0193671}. 
It is shown there that the Feller properties are preserved, if the domain is regular in the following sense.

\begin{proposition} \label{reg_E}
\textit{
The domain $E$ is regular, i.e.\ for every $z\in \partial E$ we have $P^z[\tau_E=0]=1$.
}
\end{proposition}

\begin{remark}
 \label{rem_reg}
$z\in \partial E$ is regular point in the sense of the definition given in Proposition~\ref{reg_E} if and only if for every continuous function $f$ on $\partial E$
\[
\lim_{E\ni x\rightarrow z} E^x[f(Y_{\tau_E})\indicator_{\{\tau_E <\infty \}}]=f(z).
\]
For the Brownian case we refer to Theorem 2.12 in \cite{MR1121940}. 
The arguments are robust, but it is required that under $P^x$ the probability of the event that the exit time of a ball centered at $x$ does not exceed $t$ is arbitrary  small uniformly in $x$ for a suitable chosen small $t>0$. In our situation this property is ensured by \cite{MR960535}, p.\ 330.
\end{remark}

The proof of this proposition will be based on the following Wiener
test established by Fabes, Jerison and Kenig, c.f. \cite[Thorem
5.1]{MR688024}.
\begin{theorem} \label{wiener_test}
\textit{
 Let $B_{R_0}(0)$ be a large ball centered in zero such that $E\subset B_{R_0/4}(0)$. Then, a point $z\in \partial E$ is regular if and only if
\begin{itemize}
 \item [a)] $\int_0^{R_0} \frac{s^2}{\hat q (B_s(z))} \, \frac{ds}{s} <\infty$, or
\item [b)] $\int_0^{R_0} \capacity(K_\rho) \frac{\rho^2}{\hat q (B_\rho(z))} \, \frac{d\rho}{\rho} =\infty$,
\end{itemize}
where $K_\rho(z):=(B_{R_0}(0) \backslash E)\cap B_\rho(z)$ and $\capacity$ denotes the capacity associated with the Dirichlet form $\E^{N,a}$.
}
\end{theorem}

In the sequel we will use the following notation
\[
 d(x):=\left| \left\{ i\in \{1,\ldots, N\}: x^i=x^{i-1} \right\}\right|, \qquad x\in \mathbb{R}^{N},
\]
again with the convention $x^0:=0$. In order to prove the regularity of $E$ we start with a preparing lemma.
\begin{lemma} \label{comp_parall}
\textit{
Let $y\in \mathbb{R}^{N}$ be arbitrary.}
\begin{itemize} \item[ i)] \textit{ There exists a positive $r_0=r_0(y)$ such that for all balls $B_r(y)$, $0<r\leq r_0$, there exists a parallelepiped $P_{l r}(\bar y)$  contained in $S_1$ with $d(\bar y)=d(y)$, a positive constant $C$ and $l>0$, such that}
\[
  \hat q(B_r(y)) \leq C \hat q(P_{l r}(\bar y)).
\]
\item[ii)] \textit{ For all balls $B_r(y)$ there exists parallelepiped $P_{l r}(\bar y)$  contained in $S_1$ with $d(\bar y)=d(y)$, for some $l>0$ such that  $\hat q(P_{l r}(\bar y)) \leq \hat q(B_r(y))$.}

\item[iii)] \textit{ For all $y$  we have $ C_1 \, r^{h(y)} \leq \hat q(B_r(y))$ for all $r$ and  $\hat q(B_r(y)) \leq C_2 \, r^{h(y)}$ for all $r\leq r_0(y)$
for some positive constants $C_1$ and $C_2$ depending on $y$, where}
\[
 h(y):=N-d(y)+ \tfrac \beta {N+1} \, d(y).
\]
\end{itemize}

\end{lemma}

\begin{proof}
Obviously, due to \eqref{hat_equiv_tilde} it suffices to prove the lemma with $\hat q$ replaced by $\tilde q$ defined in \eqref{def_qtilde}.

i) The case $y=0$ is clear. For $y\not= 0$ we choose $r_0$ such that $2r_0\leq |y|$ and let $n_0$ be such that $y\in S_{n_0}$. Consider now an arbitrary ball $B_r(y)$ with $r\leq r_0$ and let $K=T(S_{n_0}\cap B_r(y))$ be the subset of $S_1$ constructed in the second part of the proof of Lemma \ref{lem_parall}. Then, possibly after choosing a smaller $r_0$  we find a parallelepiped and a positive constant $l$ such that
$K\subset P_{l r}(\bar y) \subset S_1$ with $d(\bar y)=d(y)$ and we obtain i).

\noindent ii) Let $K=T(S_{n_0}\cap B_r(y))$ be defined as in i). Then, clearly $\tilde q (B_r(y)) \geq \tilde q(K)$. Thus, we can choose $\bar y=Ty$ and $l$ independent of $r$ such that $P_{l r}(\bar y)\subseteq K$ and ii) follows.

\noindent
iii) We proceed similar as in the proof of Lemma \ref{muckenhoupt}.  For some parallelepiped $P_{l r}(\bar y)$ with $d(\bar y)=d(y)$ we use a change of variables  to obtain
\begin{align*}
 \tilde q(P_{l r}(\bar y))=\int_{P_{l r}(\bar y))} \prod_{i=1}^{N} (x^i-x^{i-1})^{\frac \beta {N+1} -1} \, dx=\prod_{i=1}^{N} \int_{\tilde y_i}^{\tilde y_i+c_i r} z_i^{\frac \beta {N+1} -1} \, dz_i,
\end{align*}
where $\tilde y_i=\bar y_i - \bar y_{i-1}$,  $\bar y_0:=0$ and $c_i:=\frac l {\sqrt{N+1-i}}$ . Note that $d(\bar y)=d(y)$ is the number of components of $\tilde y$ which are equal to zero. Using the mean value theorem we obtain that
\[
 C_1 \, r^{h(y)} \leq \hat q(P_{l r}(\bar y)) \leq C_2 \, r^{h(y)}
\]
for some positive constants $C_1$ and $C_2$ depending on $y$, so that iii) follows from i) and ii).
\end{proof}

\begin{proof} [Proof of Proposition \ref{reg_E}]
 Let $z\in \partial E$ be fixed and $R_0$ be as in the statement of Theorem \ref{wiener_test}. Setting $h'(z):=1-h(z)$ let us first consider the case $h'(z)>-1$. Then, using Lemma \ref{comp_parall} iii) we have that
\[
 \int_0^{R_0} \frac{s^2}{\hat q (B_s(z))} \, \frac{ds}{s} \leq C  \int_0^{r_0} s^{h'(z)} ds + \frac 1 {2\hat q (B_{r_0}(z))} (R_0^2 -r_0^2)  < \infty,
\]
with $r_0=r_0(z)$ as above in Lemma \ref{comp_parall} iii).
Thus, the criterion a) in Theorem \ref{wiener_test} applies and the regularity of $z$ follows. The case $h'(z)\leq-1$ is more difficult. Combining Lemma 3.1 in \cite{MR643158} and Theorem 3.3 in \cite{MR643158} and using Lemma \ref{comp_parall}~iii) we get the following estimate for the capacity of small balls:
\begin{align} \label{est_cap}
 \capacity(B_r(y))\simeq \left( \int_r^{R_0}  \frac{s^2}{\hat q (B_s(y))} \, \frac{ds}{s}  \right)^{-1}\geq C \left( \int_r^{R_0}  s^{h'(y)} ds \right)^{-1}.
\end{align}
Recall the definition of the set $K_\rho(z)$ in Theorem \ref{wiener_test}. Clearly, for every $\rho$ sufficiently small there exists a ball $B_{\rho/2}(\hat z)$ with $\hat z$ depending on $\rho$ and with $d(\hat z)=d(z)$ such that $B_{\rho/2}(\hat z)\subset K_\rho(z)$.
Let now $r_0$ be such that Lemma \ref{comp_parall}~iii) and \eqref{est_cap} hold for every ball $B_\rho(z)$, $\rho\leq r_0$. Then, we obtain in the case $h'(z)< -1$
\begin{align*}
 \int_0^{R_0} \capacity(K_\rho) \frac{\rho^2}{\hat q (B_\rho(z))} \, \frac{d\rho}{\rho}
& \geq
C \int_0^{r_0} \capacity(B_{\rho/2}(\hat z)) \rho^{h'(z)} \, d\rho
\geq
C \int_0^{r_0} \left( \int_{\rho/2}^{R_0}  s^{h'(\hat z)} ds \right)^{-1} \rho^{h'(z)} \, d\rho \\
& =
C \, (h'(z)+1) \int_0^{r_0/2} \frac{\rho^{h'(z)}}{R_0^{h'(y)+1}-\rho^{h'(y)+1}} \, d\rho  \\
& =
C \int_0^{r_0/2}
\left( -\log(|R_0^{h'(y)+1}-\rho^{h'(y)+1}|) \right)^{'} d\rho=\infty.
\end{align*}
Finally, if $h'(z)=-1$ we get by an analogous procedure
\begin{align*}
 \int_0^{R_0} \capacity(K_\rho) \frac{\rho^2}{\hat q (B_\rho(z))} \, \frac{d\rho}{\rho}
& \geq
C \int_0^{r_0/2} \frac{1}{\rho (\log R_0-\log\rho)} \, d\rho \\
&=
C \int_0^{r_0/2}
\left( -\log(\log R_0-\log\rho) \right)^{'} d\rho=\infty.
\end{align*}
Hence, applying the criterion b) of Theorem \ref{wiener_test}  completes the proof.
\end{proof}

\begin{proposition} \label{feller_killed}
 \textit{
$Y^{E}$ is a Feller process, i.e.
\begin{enumerate}
\item[i)] for every $t>0$ and every $f\in C(\bar E)$ we have $Q_t^E f \in C(\bar E)$,
\item[ii)] for every $f\in C(\bar E)$, $\lim_{t\downarrow 0} Q_t^E f=f$ pointwise in $\bar E$.
\end{enumerate}
}
\end{proposition}

The statement is classical and can be found e.g.\ in Theorem~13.5 in
\cite{MR0193671}. For readability we repeat the argument here.
We shall need the following lemma. 
\begin{lemma} \label{lem_tauE}
\textit{
 For any compact set $K\subset E$ we have
\[
 \lim_{t\downarrow 0} \sup_{x\in K}  P^x[\tau_E \leq t]=0.
\]
}
\end{lemma}
\begin{proof}
We need to show that for any $\delta>0$ there exists a $t_0>0$ such that
\begin{align} \label{est_px_tauE}
\inf_{x\in K}  P^x[\tau_E\geq t_0]\geq 1-\delta.
\end{align}
Consider a bounded function $f\in C_0(\mathbb{R}^N)$ such that $0\leq f \leq 1$, $f=1$ on $K$ and $f=0$ on the complement of $E$. Let now $t_0$ be such that $\sup_{t\leq t_0} \| Q_tf -f \|_\infty < \delta/2$ (cf.\ Remark~\ref{rem_Y}). Then,
\begin{align*}
E^x[f(Y_{t_0})] \leq P^x [\tau_E \geq t_0] + E^x[f(Y_{t_0}) \, \indicator_{\{ \tau_E<t_0\} }].
\end{align*}
For $x\in K$ the left hand side is equal to
\begin{align*}
 Q_{t_0}f(x)=1+Q_{t_0}f(x)-f(x)\geq  1- \sup_{t\leq t_0} \| Q_tf -f \|_\infty \geq 1- \frac \delta 2.
\end{align*}
On the other hand, using the strong Markov property (cf.\ again Remark~\ref{rem_Y}) and the fact that $f(Y^N_{\tau_E})=0$ we have
\begin{align*}
E^x[f(Y_{t_0}) \, \indicator_{\{ \tau_E<t_0\} }]&=E^x\left[ E^x[f(Y_{t_0})\left| \mathcal{F}_{\tau_E} \right. ]  \, \indicator_{\{ \tau_E<t_0\} }\right]
=E^x\left[ Q_{t_0-\tau_E}f(Y_{\tau_E}) \, \indicator_{\{ \tau_E<t_0\} }\right] \\
&=E^x\left[ \left( Q_{t_0-\tau_E}f(Y_{\tau_E})-f(Y_{\tau_E}) \right) \, \indicator_{\{ \tau_E<t_0\} }\right]
\leq E^x\left[ \sup_{t \leq t_0} \| Q_t f-f \|_\infty \, \indicator_{\{ \tau_E<t_0\} }\right] \\
& \leq \frac \delta 2,
\end{align*}
and \eqref{est_px_tauE} follows.
\end{proof}

\begin{proof} [Proof of Proposition~\ref{feller_killed}]
Let $t>0$ and $f\in C(\bar E)$. Then, by the semigroup property of $(Q_t^E)$ we have for $0<s<t$
\[
 Q_t^Ef(x)=E^x\left[\psi_s(Y_{s\wedge\tau_E})   \right],
\]
where
\[
 \psi_s(x)=Q^E_{t-s}f(x)=E^x\left[f(Y_{(t-s)\wedge \tau_E})  \right], \qquad x\in \bar E.
\]
Then, $\psi_s$ can be extended to a function in $L^2(\mathbb{R}^N,\hat q)$ and by Proposition~\ref{Yfeller} we have $Q_s \psi_s\in C_b(\mathbb{R}^N)$. Since
\[
 \left| Q_t^Ef(x)-Q_s \psi_s(x)\right|= \left| E^x[\psi_s(Y_{s\wedge\tau_E})-\psi_s(Y_s)] \right| \leq 2 \, \| \psi_s\| \,  P^x[\tau_E\leq s]  \leq 2 \,  \| f\| \, P^x[\tau_E\leq s]
\]
and since the right hand side converges to zero uniformly in $x$ on
every compact subset of $E$ by Lemma~\ref{lem_tauE}, we conclude
that $ Q_t^Ef\in C_b(E)$, i.e.\
 $Q_t^Ef$ is continuous in the interior of $E$. In order to show i) it suffices to verify continuity at the boundary. Since $E$ is regular, we have obviously $Q_t^Ef =f$ on $\partial E$.
By Lemma~13.1 in \cite{MR0193671} we have upper semicontinuity of the mapping $x\mapsto P^x[t<\tau_E]$. Hence, we obtain for $z\in \partial E$,
\begin{align*}
 \limsup_{x\rightarrow z} P^x[t<\tau_E] \leq P^z[t<\tau_E]=0,
\end{align*}
where we have used the regularity of $E$ in Proposition~\ref{reg_E}. Thus, for every $x\in E$
\begin{align*}
 \left| Q_t^E f(x) -f(z) \right|
 \leq &
\left| Q_t^{E}f(x)-E^x[f(Y_{\tau_E}) \indicator_{\{\tau_E <\infty \}}] \right| + \left| E^x[f(Y_{\tau_E}) \indicator_{\{\tau_E <\infty \}}]-f(z) \right| \\
\leq & \left| E^x\left[ f(Y_t) \indicator_{\{ t<\tau_E\}} + f(Y_{\tau_E}) \indicator_{\{ t \geq \tau_E\}}-f(Y_{\tau_E}) \indicator_{\{\tau_E <\infty \}} \right] \right| \\
&+\left| E^x[f(Y_{\tau_E}) \indicator_{\{\tau_E <\infty \}}]-f(z) \right|\\
\leq &\|f\|_\infty P^x[t<\tau_E]+   \left| E^x\left[ f(Y_{\tau_E}) \indicator_{\{\tau_E <\infty \}} \left( \indicator_{\{ t \geq \tau_E\}}-1 \right) \right] \right| \\
 \leq &  2 \|f\|_\infty P^x[t<\tau_E]+\left| E^x[f(Y_{\tau_E})\indicator_{\{\tau_E <\infty \}}]-f(z) \right|,
\end{align*}
where we have used the fact that the event $\{\tau_E<\infty \}$ is included in $\{ t\geq \tau_E\}$.
Since $z$ is a regular point, the second term tends to zero as  $x\rightarrow z$, $x\in E$, cf.\ Remark~\ref{rem_reg}. Hence,
\[
 \lim_{E\ni x\rightarrow z}Q_t^E f(x)=f(z)
\]
and property i) is proven.

To prove ii), we extend $f$ to a function in $C_0(\mathbb{R}^N)$, i.e.\ we may deduce from Proposition~\ref{Yfeller} that $\lim_{t\downarrow 0} Q_t f(x)=f(x)$ for every $x\in E$. Furthermore, we have for every $x\in E$
\[
 | Q_t^E f(x)-Q_t f(x) | \leq  P^x[t \leq \tau_E] \, \|f\|_\infty,
\]
and $P^x[\tau_E=0]=0$, since $E$ is open and $Y$ has continuous paths. Hence,
\[
\lim_{t\downarrow 0} Q^E_t f(x)= \lim_{t\downarrow 0} Q_t f(x)=f(x).
\]
This gives pointwise convergence on $E$. Since $Q_t^Ef =f$ on $\partial E$ by regularity, the convergence on $\partial E$ is trivial.
\end{proof}

\subsubsection{Feller Properties of $X^N$} In this section we will
finally prove the Feller property for $X^N$ stated in
Theorem~\ref{fellerthm}.  To this end we construct a Feller process
$\tilde X$ taking values in $\overline{\Sigma}_N$ and in a second
step we will identify this process with $X^N$.
\smallskip

In analogy to the definition of $\Omega_N$ above we set
\begin{align}\label{omegaidef}
\Omega_i:= \Omega_i^\delta:=\{ x\in \overline{\Sigma}_N:\, x^{i+1}-x^i \geq \delta \}, \qquad i=0,\ldots,N.
\end{align}
and moreover
\[
 A^i:=\partial \Omega^{2\delta}_i \backslash \partial \Sigma_N=\{ x\in\overline{\Sigma}_N:\, x^{i+1}-x^i=2 \delta\}.
\]
Notice that we can choose $\delta$ so small that $\overline{\Sigma}_N=\bigcup_{i=1}^N \Omega^{2\delta}_i$.
Furthermore, we define the mappings $H_i$, $i=0,\ldots,N$, by
\begin{align*}
 H_i(x):=\left( x^1,\ldots, x^i, 1-(x^N-x^i), 1-(x^{N-1}-x^i), \ldots, 1-(x^{i+1}-x^i) \right), \qquad x\in \overline{\Sigma}_N.
\end{align*}
Notice that for every $i$, $H_i$ maps $\Omega_N$ on $\Omega_i$ and
vice versa. In particular, $H_i \circ H_i=\id$ and $H_N$ is the
identity on $\Omega_N$. Let $T: \bar E\rightarrow \Omega_N^{2\delta}\subset\Omega_N$ be
defined as in \eqref{def_T}. Let $Y^i$ denote the $\R^N$-valued
Feller process induced from the form \eqref{globaldirichletform}
with $a= a_i$, where the matrix valued function $ a$ is defined
$\hat q$-almost everywhere by
\begin{equation}
 a_i(x):= \left\{
\begin{array}{ll}
 DH_i^{\rm t}  & \mbox{for } x \in \mathring S_1 \\
DH_i^{\rm t} \cdot (DT_{|S_n}^{-1})^{\rm t} &  \mbox{for } x \in \mathring S_n,
\,n\in \{2, \cdots 2^N\cdot N!\},
 \end{array} \right.
\label{matrixchoice} \end{equation}
 such that condition \eqref{ass_a} is clearly satisfied (note that $DH_i=DH_i^{-1}$). In particular, $a_i$ is constant on every $S_n$. Moreover, setting $\rho_i:=H_i\circ T$, clearly $a_i=(D\rho_i^{-1})^{\rm t}$.

\begin{remark} \label{ibpf_Y}
 Analogously to Proposition~\ref{ibpf} one can establish the following integration by parts formula associated to $\E^{N,a_i}$. Let $u\in C^1(\R^N)$ and $\xi$ a continuous vector field on $\R^N$ such that $\supp \xi \subseteq \{ x\in\R^N: \|x\|_\infty <1-\delta\}$, $\xi$ is continuously differentiable in the interior of each $S_n$  and for every $n$ $\xi$ satisfies the boundary condition $\langle a_i^t \cdot a_i \cdot \xi, \nu_n \rangle=0$ on $\partial S_n$, $\nu_n$ denoting the outward normal field of $\partial S_n$. Then,
\[
 \int_{\R^N} \langle a_i \nabla u, a_i \xi \rangle \hat q(dx)=-\int_{\R^N} u \left[ \divergence(a_i^t\cdot a_i \cdot \xi)+\langle a_i^t\cdot a_i \cdot \xi,\nabla \log \hat q  \rangle \right] \hat q(dx).
\]
Thus, every smooth function $g$ such that $\xi=\nabla g$ satisfies the above conditions is contained in the domain of the generator $L^i$ associated to $\E^{N,a_i}$
and on $\R^N \setminus \bigcup_n \partial S_n$ we have
\[
 L^i g=\divergence(a_i^t\cdot a_i \cdot \nabla g)+\langle a_i^t\cdot a_i \cdot \nabla g,\nabla \log \hat q  \rangle .
\]

\end{remark}

Next we
 define the $\Omega_i^{2\delta}$-valued process $\tilde X^{i}$ by
\[
\tilde X^{i}_t=\rho_i(Y_t^{i,E}) =H_i\circ T(Y_t^{i,E}), \qquad t\geq 0.
\]
The semigroup of $\tilde X^i$ will be denoted by $(\tilde
P^i_t)_{t\geq 0}$, i.e.\ $ \tilde P^i_tf(x)=E^x[f(\tilde X^i_t)]$.

\begin{lemma}
\textit{
 For every $i$, $\tilde X^{i}$ is Markovian.
}
\end{lemma}
\begin{proof}
Since $H_i$ is an injective mapping for every $i$, it suffices to
show that the process $T(Y^{i,E}_\cdot)$ is Markovian. Moreover,
since $T(Y^{i,E}_t) =  T(Y^i_{t\wedge \tau_E}) = (T(Y^i_\cdot))_{t
\wedge \tau_E}$ it is enough to  prove the Markov property for the
process $T(Y^i_\cdot)$, which is implied e.g.\  by the condition
that for any Borel set $A\subseteq \overline{\Sigma}_N$
\begin{align}
 P^x[Y_t^{i}\in T^{-1}(A)]=P^y[Y_t^{i}\in T^{-1}(A)] \qquad \text{whenever
 $T(x)=T(y)$}.
\label{markovcondition}
\end{align}
Now the choice of $\hat q$ and the metric $a_i$ imply for any
Borel set $A\subseteq \overline{\Sigma}_N$ condition \eqref{markovcondition} is
satisfied. To see this,  let  $\{\sigma_k\, |\, k = 1, \cdots, N +
N(N-1)/2\}$ be the collection of line-reflections in $\R^N$ with
respect to either one of the coordinate axes $\{ \lambda e_i,
\lambda \in \R\}$ or a diagonal $\{\lambda (e_j+e_k)\}$, then for
$x,y \in \R^N$ with $T(x)=T(y)$ there exists a finite sequence
$\sigma_{k_1},\cdots \sigma_{k_l}$ such that $\tau(x):=\sigma_{k_1}\circ
\sigma_{k_2} \cdots \circ \sigma_{k_l} (x) = y$. Now each of the
reflections $\sigma_i$ preserves the Dirichlet form
\eqref{globaldirichletform} when $a$ is chosen as in
\eqref{matrixchoice}, such that the processes $\tau(Y^{i,x}_\cdot )$
and  $Y^{i,y}_\cdot$ are equal in distribution. Moreover, $\tau(T^{-1}(A)) = T^{-1}(A)$, from which \eqref{markovcondition} is
obtained.\end{proof}

\begin{lemma} \label{localmartp}
\textit{For each $f\in C^2_{Neu}$ the process $t \to f(\tilde
X^i_t) - \int_0^t L^Nf(\tilde X^i_s)\, ds$ is a martingale w.r.t.\ to the
filtration generated by $\tilde X^i_\cdot $.}

\begin{proof} Similar to Proposition  \ref{ibpf}  one
checks for $f\in C^2_{Neu} \cap C_c(\Omega_i)$ that the function $f_i$ on $\R^N$, which is defined by $f_i =
f\circ H_i \circ T=f\circ \rho_i$ on the set $ \{ x\in\R^N: \|x\|_\infty <1-\delta\}$ and $f_i=0$ on the complement this set, belongs to the domain of the
generator $L^i$ of the Dirichlet form \eqref{globaldirichletform}
with $a=a_i$ as in \eqref{matrixchoice} (cf.\ Remark~\ref{ibpf_Y}). Hence the process $f_i
(Y^i_\cdot  ) - \int_0^\cdot L^i f_i (Y^i_s) ds $ is a martingale
w.r.t. to the filtration generated by $Y^{i}$ and thus also $f_i
(Y^{i,E}_\cdot  ) - \int_0^\cdot L^i f_i (Y^{i,E}_s) ds $ due to the
optional sampling theorem.
Obviously in the last statement the
function $f$ can be modified outside of $\Omega_i^{2\delta}$, i.e. it holds
also for $f\in C^2_{Neu}$.
Moreover, $f_i(Y^{i,E}_\cdot) =
f(\tilde X^i _\cdot)$ and $L^i f_i = (L^N f) \circ \rho_i$ on $\bar E$.
Thus, $f(\tilde X^i_\cdot ) - \int_0^\cdot L^N f(\tilde X^i_s)
ds $ is a martingale w.r.t.\ the filtration generated by $Y^i _\cdot
$ and adapted to the filtration generated by $\tilde X^{i}_\cdot=
H_i \circ T (Y^{i,E}_\cdot)$ which establishes the claim.
\end{proof}

\end{lemma}

\begin{proposition}
\textit{ For every $i$, $\tilde X^{i}$ is a Feller process, more precisely
\begin{enumerate}
\item[i)] for every $t>0$ and every $f\in C(\Omega_i^{2\delta})$ we have $\tilde P^i_t f \in C(\Omega_i^{2\delta})$,
\item[ii)] for every $f\in C(\Omega_i^{2\delta})$, $\lim_{t\downarrow 0} \tilde P^i_t f=f$ pointwise in $\Omega_i^{2\delta}$.
\end{enumerate}
}
\end{proposition}
\begin{proof}
 Since obviously for every continuous $f$ on $\Omega_i^{2\delta}$
\[
 \tilde P^i_tf(x)=E^x[f(\tilde X^i_t)]=E^x[f(H_i\circ T(Y^{i,E}_t))]=Q^E_t(f \circ H_i\circ T)(x), \qquad t>0,
\]
the result follows from Proposition~\ref{feller_killed} and the
continuity of the mappings $H_i$ and $T$.
\end{proof}

Next we define the process $\tilde X$ with state space $\overline{\Sigma}_N$ as follows: Let $q_N$ be the initial distribution of $\tilde X$, i.e.\  $\tilde X_0 \sim q_N$. Choose $i_1\in \{0,\ldots N\}$ such that $\tilde X_0 \in \Omega_{i_1}^{2\delta}$ and $\dist(\tilde X_0,A^{i_1})=\max_i \dist(\tilde X_0,A^{i})$. We set $\tilde X_t=\tilde X_t^{i_1}$ for $0\leq t \leq T_1$, where $T_1$ denotes the first hitting time of $A^{i_1}$, i.e.\ on $[0,T_1]$ the process behaves according to $\tilde P^{i_1}$. Choose now $i_2\in \{0,\ldots N\}$ such that $\tilde X_{T_1} \in \Omega_{i_2}^{2\delta}$ and $\dist(\tilde X_{T_1},A^{i_2})=\max_i \dist(\tilde X_{T_1},A^{i})$. At time $T_1$ the process starts afresh from $\tilde X_{T_1}$ according to $(\tilde P^{i_2}_t)$ up to the first time $T_2$ after $T_1$, when it hits $A^{i_2}$. This procedure is repeated forever.

\begin{proposition}
 \textit{$\tilde X$ is a Feller process.}
\end{proposition}
\begin{proof}
 Let $(\tilde P_t)$ denote the semigroup associated to $\tilde X$. For $f\in C(\overline{\Sigma}_N)$, we need to show that $\tilde P_t f\in C(\overline{\Sigma}_N)$ for every $t>0$.
Let us first show that $\tilde P_{T_n} f\in C(\overline{\Sigma}_N)$ for every $n$ using an induction argument. For an arbitrary $x\in \overline{\Sigma}_N$, choose $i_1$ as above depending on $x$ such that $\tilde P_{T_1} f(x)=\tilde P^{i_1}_{T_1} f(x)$. Since $\tilde P^{i_1}_{T_1} f \in C(\Omega_{i_1}^{2\delta})$, we conclude that $\tilde P_{T_1} f$ is continuous in $x$ for every $x$.
For arbitrary $n$ and $x\in \overline{\Sigma}_N$ we have by the strong Markov property
\begin{align*}
 \tilde P_{T_{n+1}} f(x)&=E^x[f(\tilde X_{T_{n+1}})]=E^x\left[E_{\tilde X_{T_N}}[f(\tilde X^{i_n}_{T_{n+1}-T_n})] \right]=E^x\left[ \tilde P^{i_n}_{T_{n+1}-T_n}f(\tilde X_{T_N}) \right] \\
&=\tilde P_{T_{n}} (P^{i_n}_{T_{n+1}-T_n} f)(x)
\end{align*}
and since $P^{i_n}_{T_{n+1}-T_n} f$ can be extended to a continuous function on $\overline{\Sigma}_N$, we get  $\tilde P_{T_{n+1}} f\in C(\overline{\Sigma}_N)$ by the induction assumption.

Similarly, one can show that for every $n$ the mapping $x\mapsto E^x[f(\tilde X_t) \indicator_{\{ t\in (T_n,T_{n+1}] \}}]$ is continuous. Finally, for every $x\in \overline{\Sigma}_N$
\begin{align*}
 \left| \tilde P_t f(x) - \sum_{k=0}^{n-1}  E^x\left[f(\tilde X_t) \indicator_{\{ t\in (T_k,T_{k+1}] \}}\right]\right|=\left| E^x\left[f(\tilde X_t) \indicator_{\{ t>T_n \}}\right] \right|\leq \|f\|_\infty \, P^x[t>T_n]
\end{align*}
and since $T_n \nearrow \infty$ $P^x$-a.s. locally uniformly in $x$
as $n$ tends to infinity, the claim follows. \end{proof}

The proof of Theorem~\ref{fellerthm} is complete, once we have shown that the processes $\tilde X$ and the original process $X^N$ have the same law.

\begin{lemma} \textit{For $x\in\overline{\Sigma}_N$ the process $(\tilde X^x _\cdot )$ obtained from conditioning
 $\tilde X$ to start in $x$, solves the martingale problem for the operator
 $(L^N,C^2_{Neu})$ as in \eqref{generator} and starting distribution
 $\delta_x$.}
\end{lemma}

\begin{proof}
Let $\{T_k\}$ be the sequence of
strictly increasing stopping times introduced in the construction of
the process $\tilde X$. Then for $s< t$
\[ f(\tilde X_t)-f(\tilde X_s) -\int_s^t L^Nf(\tilde X_\sigma)d\sigma = \sum_k
f(\tilde X_{(T_{k+1}\vee s)\wedge t}) - f(\tilde X_{(T_{k}\vee
s)\wedge t}) -\int\limits_{(T_{k}\vee s)\wedge t}^{(T_{k+1}\vee
s)\wedge t} L^Nf(\tilde X_\sigma)d\sigma.
\]
Hence
\begin{align*}
\mathbb E \bigl( f(\tilde X_t ) &- \int_0^t L^Nf(\tilde X_\sigma)
d\sigma \, \bigl|\mathcal F_s\bigr)   = f(\tilde X_s ) - \int_0^s
L^Nf(\tilde X_\sigma)
d\sigma  \\
& + \sum_k \mathbb E \bigl( f(\tilde X_{(T_{k+1}\vee s)\wedge t}) -
f(\tilde X_{(T_{k}\vee s)\wedge t}) -\int\limits_{(T_{k}\vee
s)\wedge t}^{(T_{k+1}\vee s)\wedge t} L^Nf(\tilde
X_\sigma)d\sigma\,\bigl| \mathcal F_s \bigr).
\end{align*}
Using  the strong Markov property of the Feller process $\tilde
X_\cdot$ and its pathwise decomposition into pieces of $\{\tilde
X^{i}\}$-trajectories one obtains
\begin{align}
\mathbb E \bigl(  & f(\tilde X_{(T_{k+1}\vee s)\wedge t}) - f(\tilde
X_{(T_{k}\vee s)\wedge t})   -\int\limits_{(T_{k}\vee s)\wedge
t}^{(T_{k+1}\vee s)\wedge t} L^Nf(\tilde X_\sigma)d\sigma\,\bigl|
\mathcal F_s \bigr) \nonumber \\
 & = \mathbb E \bigl[ \mathbb E_{\tilde X_{(T_k\vee s)\wedge t}} \bigl( f(\tilde
 X^{i_k}_{\tau\wedge (t-s)}) -f(\tilde
 X^{i_k}_0)
 - \int_0^{\tau\wedge (t-s)} L^Nf(\tilde X^{i_k}_\sigma) d\sigma\bigr )\, \bigl|
 \mathcal F_s \bigr) \bigr] \label{expectzero},
 \end{align}
where, by construction of $\tilde X_\cdot$,  $\tau$ ist the hitting
time of the the set $A_{i_k}$ for which  $\dist(A_{i_k},\tilde
X_{T_k}) = \max_i \dist(A_i,\tilde X_{T_k})$. Lemma \ref{localmartp}
implies that the inner expectation in \eqref{expectzero} is zero.
\end{proof}

The last ingredient for our proof of Theorem \ref{fellerthm} is the
identification of the processes $\tilde X_\cdot$ with $X_\cdot$.
Since both are Markovian and solve the martingale problem for the
operator $(L^N,C^2_{Neu})$ it suffices to show that the martingale
problem admits at most one Markovian solution. Clearly,  any such
solution induces a symmetric sub-Markovian semigroup on
$L^2(\Sigma_N,q_N)$ whose generator extends $(L^N,C^2_{Neu})$. Hence
it is enough
 to establish the following so-called Markov uniqueness property of
$(L^N,C^2_{Neu})$, cf.\ \cite[Definition 1.2]{MR1734956}, stated in Proposition~\ref{uniqueness}.


\begin{proof}[Proof of Proposition~\ref{uniqueness}]
Let $H^{1,2}(\Sigma_N, q_N)$  (resp.\ '$H^{1,2}_0(\Sigma_N,q_N)$' in
the notation of \cite{MR1734956}) denote the closure of
$C^2_{Neu}$ w.r.t. to the norm $\norm{f}_1 =  ({
\norm{f}^2_{L^2(\Sigma_N, q_N)} + \norm{\nabla f}^2_{L^2(\Sigma_N,
q_N)}})^{1/2}$ and let $W^{1,2}$  be the Hilbert space of
$L^2(\Sigma_N, q_N)$-functions $f$ whose distributional derivative
$Df$ is in $L^2(\Sigma_N,q_N)$, equipped with the norm $\norm{f}_1 =
( { \norm{f}^2_{L^2(\Sigma_N, q_N)} + \norm{D f}^2_{L^2(\Sigma_N,
q_N)}})^{1/2}$. Clearly, the quadratic form $Q(f,f) = \langle Df,
Df\rangle_{L^2(\Sigma_N, q_N)}$, $\mathcal D(Q) = W^{1,2}(\Sigma_N,
q_N)$,  is a Dirichlet form on $L^2(\Sigma_N,q_N)$. Hence we may use
the  basic criterion for Markov uniqueness \cite[Corollary
3.2]{MR1734956}, according to which
 $(L, C^2_{Neu})$ is
Markov unique if $H^{1,2}=W^{1,2}$. \smallskip

To prove the latter it obviously suffices to prove  that $H^{1,2}$
is dense in $W^{1,2}$. Again we proceed by localization as follows.
For fixed $\delta>0$ let $\Omega_i^\delta  = \Omega_i$ denote the
subsets defined in \eqref{omegaidef}, then $\{\Omega^{3\delta}_i\}$
constitutes a relatively open covering of $\overline{\Sigma}_N$ for $\delta $
small enough. Let $\{\eta_i\}_{i=0,\cdots, N}$ and
$\{\chi_i\}_{i=0,\cdots, N}$ be two smooth partitions of unity on
$\overline{\Sigma}_N$ such that $\eta_i= 1 $ on $ \Omega_i^{3\delta}$ and
$\supp(\eta_i) \subset \Omega_i^{2\delta}$ and $\chi_i= 1 $ on $
\Omega_i^{2\delta}$ and $\supp(\chi_i) \subset \Omega_i^\delta$
respectively. Hence writing  $f \in W^{1,2}(\Sigma_N, q_N)$ as $f =
\sum_i f_i$ with $f_i = \eta_i \cdot f \in W^{1,2}(\Sigma_N, q_N)$
it suffices to prove that $f_i \in H^{1,2}(\Sigma_N, q_N)$. We show
first that $f_i$ can be approximated w.r.t.\ $\norm{.}_1$ by
functions which are smooth up to boundary and secondly that such
functions can be approximated in $\norm{.}_1$ by smooth Neumann
functions. We give details for the case $i=N$ only, the other cases
can be treated almost the same way by using the maps $H_i$.

\smallskip
Let  $g_N = f\cdot \chi_N \in W^{1,2}(\Sigma_N, q_N)$, then the
restriction of $g_N $ to $\Omega_N^\delta$ belongs to the space
$W^{1,2}(\Omega_N^{\delta},q_N)= W^{1,2}(\Omega_N^{\delta},\hat q)$,
where $\hat q$ denotes the modification of $q_N$ according to
\eqref{def_hatqn}. Due  to Lemma \ref{muckenhoupt} the extension
$\hat q(x) = \hat q(T(x))$, $x\in   \R^N$, lies in the Muckenhoupt
class $\mathcal A_2$. Further note that
$W^{1,2}(\Omega_N^{\delta},\hat q) \subset
W^{1,2}(\Omega_N^{\delta},dx)=H^{1,2}(\Omega_N^{\delta},dx)$ if $ \beta \leq N+1$, and $W^{1,2}(\Omega_N^{\delta},\hat q) \subset
W^{1,1}(\Omega_N^{\delta},dx)=H^{1,1}(\Omega_N^{\delta},dx)$ by the H\"older inequality if $ N+1 < \beta <2 (N+1)$,   such
that $g_N$ has well defined boundary values in $L^1(\partial
\Omega_N, dx)$. Hence we may conclude that the extension $\hat
g_N(x) = g_N(T(X))$, $x\in \R^N$ defines a weakly differentiable
function on $\R^N$ with $\norm{\hat g_N}_{W^{1,2}(\R^N, \hat q)} =
2^N \cdot N! \, \norm{ g_N}_{W^{1,2}(\Omega_N, \hat q)}$. By \cite[Theorem
2.5]{MR1246890} the mollification with the standard mollifier yields
an approximating sequence $\{u_l\}_l$ of smooth functions $u_l\in
C^\infty_0(\R^N)$ of $g_N$ in the weighted Sobolev spaces $H^{1,2}(
\R^N ,\hat q)$. Extending each $u_l\cdot \eta_N$ by zero in
$\overline{\Sigma}_N\setminus \Omega^{\delta}_N$ we obtain a sequence of
$C^\infty(\overline{\Sigma}_N)$-functions $\{u_l\cdot \eta_N\}_l$ which
converges to
  $g_N \cdot \eta_N = f_N$ in $H^{1,2}(\Sigma_N,q_N)$.
  This finishes the
  first step. In the second step we thus may assume w.l.o.g.\
that  $g_N $ is smooth up to the boundary of $\Sigma_N$. In
particular, $g_N$ is globally Lipschitz. Since $q_N$ is integrable
on $\Sigma_N$ we may modify $f_N$ close to the boundary to obtain a
Lipschitz function $\tilde f_N$ which satisfies the Neumann
condition   and which is close to $f_N$ in $\norm{.}_1$. (Take,
e.g.\ $\tilde f_N(x) = f(\pi(x))$, where $\pi(x)$ is the projection
of $x$ into the set $\Sigma_N^{\epsilon}= \{ x\in \Sigma_N, \, |
\dist(x,\partial \Sigma_N)\geq \epsilon\}$ for small
 $\epsilon>0 $.)  We may now proceed as in step one to obtain an
approximation of $\tilde f_N$ by smooth functions w.r.t.
$\norm{.}_1$, where we
 note that neither extension by reflection through the map $T$ nor
 the standard mollification  in \cite{MR1246890} of the extended $\tilde f_N$ destroys the Neuman
 boundary condition.
\end{proof}

Hence we arrive at the following statement which implies the first
statement of Theorem \ref{fellerthm} for the cases when  $\beta
< 2(N+1)$.

\begin{corollary}
 \textit{For quasi-every $x \in \overline{\Sigma}_N$, the
the processes $\tilde X^x_\cdot$ and $X^x_\cdot$ are equal in law.
In particular, $X$ is a Feller process on $\overline{\Sigma}_N$.}
\end{corollary}

\subsection{The case $\beta \geq (N+1)$.} The estimate
\eqref{contraction} looks (also for the case $\beta < N+1$) like a
straightforward application of the Bakry-Emery $\Gamma_2$-calculus.
However,  the complete Bakry-Emery criterion requires  the
$\Gamma_2$-condition on an algebra of functions which is dense in
the domain of the generator $\mathcal L$ of $\E$ w.r.t.\ the graph
norm. The verification of the latter typically leads back to
elliptic regularity theory for $\mathcal L$ which we want to avoid.
Instead our argument below is based on a recent powerful result by
Ambrosio, Savare and Zambotti \cite{asz}
on the stability of reversible processes with log-concave invariant measures.\\

\textit{Proof of  \eqref{contraction}.} By abuse of notation let
$q_N$ be the extension of the measure $q_N$ to $\R^N$ given by
\[q_N ( A)
= q_N(A\cap \Sigma_N)= \frac 1 {Z_\beta}\int_{A\cap\,  \{V<
\infty\}} e^{- (\frac\beta  {N+1} -1) V(x) } dx
\] for any Borel set $A\subset \R^N$, where $V: \R^N \to \R\cup\{\infty\}$,
\[ V(x) = \left\{ \begin{array}{ll} - \sum_{i=1}^{N+1}
\ln(x_i-x_{i-1}) &  x\in \Sigma_N\\
\infty & x \in \R^N \setminus \Sigma_N \end{array} \right.\] is a
lower semicontinuous and convex function on $\R^N$. In fact, by
elementary calculations, for $x \in \Sigma_N$ with $\rho_i=
\rho_i(x) = 1/(x_i-x_{i-1})^2$, $i= 1,\dots, {N+1}$,
\[ \langle \xi , \nabla^2 V(x) \xi\rangle = \rho_1 \xi_1^2 + \sum_{i=1}^{N-1}
\rho_{i+1} (\xi_{i+1}-\xi_i)^2 + \rho_{N+1} \xi_{N}^2\geq k_N
|\xi|^2, \quad \forall \xi \in \R^N,
\]
where $k_N>0$ denotes the smallest eigenvalue of the strictly
elliptic matrix $A=( 2 \delta_{ij} - \delta_{1,|i-j|}) \in
\R^{\N\times\N}$. Hence the measure $q_N$ is log-concave in the
sense of \cite{asz}. (Note that according to ibid.\ Theorem 1.2.b.\
the Feller property of $X_.$ on $\ol \Sigma_N$ is automatically
implied.) Let $\overline x \in \Sigma_N$ denote the barycenter of
the simplex $\Sigma_N$. For $\epsilon
>0$ let
\[ V_\epsilon(x)= \left\{
\begin{array}{ll}
 V((1-\epsilon) x + \epsilon \overline x )) & x\in \overline \Sigma_N \\
 \infty & x\in \R^N \setminus \overline \Sigma_N
\end{array}
\right.\] and define the measure $q_N^\epsilon$ on $\R^N$ by
\[q^\epsilon_N(A) =\frac{1 }{Z_\beta} \int_{A\cap \Sigma_N}
e^{ - (\frac \beta {N+1}-1) V_\epsilon(x)}  dx.\] Since
$V_\epsilon\in C^\infty(\Sigma_N)$ and the boundary of $\Sigma_N$
is piecewise smooth this process can also be constructed by the
corresponding Skorokhod SDE. Moreover, the classical coupling be
reflection method can be applied, c.f. \cite{MR1262968}. Taking
expectations in \cite[eq. (2.5)]{MR1262968} (note that $I=0$ in our
case) and using the strict convexity of $V_\epsilon$ together with
Gronwall's lemma this yields the estimate
\[ \mathbb E_{C}\bigl (\overline X^{\epsilon,1}_t,\overline X^{\epsilon,2}_t\bigr)
\leq e^{-(\frac \beta {N+1} -1) (1-\epsilon)^2 k_N t} d(x,y),\]
where $C $ denotes the law of the coupling process $(\overline
X^{\epsilon,1}_t,\overline X^{\epsilon,2}_t)_{t\geq 0}$, starting in
$(x,y)$. In particular the following estimate  in the
$L^1$-Wasserstein distance $d_{1}$  for the heat kernel of
$(X^\epsilon_.)$ is obtained
\[ d_1(P^\epsilon_t(x,\cdot), P^\epsilon_t(y,\cdot)) \leq e^{-(\frac \beta {N+1} -1) (1-\epsilon)^2 k_N t}
d(x,y).\] Since $q_N^\epsilon \to q_N$ weakly, we may now invoke
Theorem 6.1. of \cite{asz} in order to pass to the limit for
$\epsilon \to 0$ in the left hand side above, using also the
continuity of the $L^1$-  w.r.t\ the $L^2$-Wassserstein metric.
Hence we arrive at
\[ d_1(P_t(x,\cdot), P_t(y,\cdot)) \leq e^{-(\frac \beta {N+1} -1)
(1-\epsilon)^2 k_N t} d(x,y),\] for all $\epsilon >0$ small enough,
thus  also for $\epsilon =0$. Via Kantorovich duality this implies
\[ |P_t f(x) - P_t f(y)| \leq e^{-(\frac \beta {N+1}-1 )  k_N t}
\Lip(f) \, d(x,y),\] for or all $f\in \Lip(\Sigma_N)$  and $x,y \in
\Sigma_N$, which is  the claim. \bbox

\section{Semi-Martingale Properties} In this
final section we prove the semi-martingale properties of $X^N$
stated in Theorem \ref{thmsemimart}. To that aim we establish the
semi-martingale properties for the symmetric process $X^N$ started
in equilibrium, which imply the semi-martingale properties to hold
for quasi-every starting point $x\in \overline{\Sigma}_N$ and by the
Feller properties proven in the last section for every starting
point $x\in\overline{\Sigma}_N$.  In order to establish the
semi-martingale properties of the stationary process, we shall use
the following criterion established by Fukushima in
\cite{MR1741537}.  For every open set $G\subset \overline{\Sigma}_N$
we set
\[
 \mathcal{C}_G:=\{ u\in D(\E^N)\cap C(\overline{\Sigma}_N): \supp(u) \subset G \}.
\]

\begin{theorem}
\textit{
For $u\in D(\E^N)$ the additive functional $u(X^N_t)-u(X^N_0)$ is a semi-martingale if and only if
one of the following (equivalent) conditions holds:
\begin{enumerate}
\item[i)]
For any relatively compact open set $G\subset \overline{\Sigma}_N$, there is a positive constant $C_G$ such that
\begin{align} \label{cond_semimart}
 |\E^N(u,v)| \leq C_G \, \| v\|_\infty, \qquad \forall v\in \mathcal{C}_G.
\end{align}
\item[ii)]
There exists a signed Radon measure $\nu$ on $\overline{\Sigma}_N$ charging no set of zero capacity such that
\begin{align} \label{cond_semimart2}
\E^N(u,v)=-\int_{\overline{\Sigma}_N} v(x) \, \nu(dx), \qquad \forall v\in C(\overline{\Sigma}_N)\cap \D(\E^N).
\end{align}
\end{enumerate}
}
\end{theorem}
\begin{proof}
 See Theorem 6.3 in \cite{MR1741537}.
\end{proof}

\begin{theorem} \textit{
Let  $X^N$ be a symmetric diffusion on $\overline{\Sigma}_N$ associated with the Dirichlet form $\E^N$, then $X^N$ is  an $\R^{N}$-valued semi-martingale if and only if $\beta/(N+1)\geq 1$.
}
\end{theorem}

\begin{proof}
 Since the semi-martingale property for $\mathbb{R}^N$-valued diffusions is defined componentwise, we shall apply Fukushima's criterion for $u(x)=x^i$, $i=1,\ldots,N$.

Let us first consider the case where $\beta':=\beta/(N+1)>1$. Then, for a relatively compact open set $G\subset \overline{\Sigma}_N$ and $v\in \mathcal{C}_G$,
\begin{align*}
 \E^N(u,v)=& \int_{\Sigma_N} \frac \partial {\partial x^i} v(x) \, q_N(dx) \\
=& \frac 1 {Z_\beta} \int_0^1 dx^1 \int_{x^1}^1 dx^2 \cdots  \int_{x^{i-1}}^1 dx^{i+1} \int_{x^{i+1}}^1 dx^{i+2} \cdots \int_{x^{N-1}}^1 dx^N \prod_{\substack{j=0\\ j\not= i-1,\,i}}^N (x^{j+1}-x^j)^{\beta'-1}  \\
& \times \int_{x^{i-1}}^{x^{i+1}} \frac \partial {\partial x^i} v(x) \, (x^{i}-x^{i-1})^{\beta'-1} \, (x^{i+1}-x^i)^{\beta'-1} \, dx^i.
\end{align*}
Since $\beta'>1$, we obtain by integration by parts
\begin{align*}
& \int_{x^{i-1}}^{x^{i+1}} \frac \partial {\partial x^i} v(x) \, (x^{i}-x^{i-1})^{\beta'-1} \, (x^{i+1}-x^i)^{\beta'-1} \, dx^i \\
=&-(\beta'-1)  \int_{x^{i-1}}^{x^{i+1}} v(x) \left[(x^{i}-x^{i-1})^{\beta'-2} \, (x^{i+1}-x^i)^{\beta'-1} - (x^{i}-x^{i-1})^{\beta'-1} \, (x^{i+1}-x^i)^{\beta'-2} \right] \, dx^i
\end{align*}
so that
\begin{align*}
 &\left| \int_{x^{i-1}}^{x^{i+1}} \frac \partial {\partial x^i} v(x) \, (x^{i}-x^{i-1})^{\beta'-1} \, (x^{i+1}-x^i)^{\beta'-1} \, dx^i \right| \\
\leq &(\beta'-1) \|v\|_\infty \left( \int_{x^{i-1}}^{x^{i+1}} (x^{i}-x^{i-1})^{\beta'-2} \, dx^i +\int_{x^{i-1}}^{x^{i+1}} (x^{i+1}-x^{i})^{\beta'-2} \, dx^i  \right) \\
\leq & 2 \, (\beta'-1) \|v\|_\infty  \int_0^1 r^{\beta'-2} dr \leq C \, \|v\|_\infty,
\end{align*}
and we obtain that condition \eqref{cond_semimart} holds. For
$\beta'=1$   the measure $q_N$ coincides with the normalized
Lebesgue measure on $\overline{\Sigma}_N$,  condition
\eqref{cond_semimart} follows directly.

Let now $\beta'<1$ and let us assume that $u(X^N_t)-u(X^N_0)$ is a semi-martingale. Then, there exists a signed Radon measure $\nu$ on $\overline{\Sigma}_N$ satisfying \eqref{cond_semimart2}. Let $\nu=\nu_1-\nu_2$ be the Jordan decomposition of $\nu$, i.e.\ $\nu_1$ and $\nu_2$ are positive Radon measures. By the abobe calculations we have
for each relatively compact open set $G\subset \overline{\Sigma}_N$ and for all $v\in \mathcal{C}_G$
\begin{align*}
\E^N(u,v)=&- \frac 1 {Z_\beta} (\beta'-1)  \int_{G} v(x)  \prod_{\substack{j=0\\ j\not= i-1,\,i}}^N (x^{j+1}-x^j)^{\beta'-1}  \\
& \times  \left[(x^{i}-x^{i-1})^{\beta'-2} \, (x^{i+1}-x^i)^{\beta'-1} - (x^{i}-x^{i-1})^{\beta'-1} \, (x^{i+1}-x^i)^{\beta'-2} \right]  \, dx.
\end{align*}
Hence,  we obtain for the Jordan decomposition $\nu=\nu_1-\nu_2$ that
\begin{align*}
\nu_1(G)&=\frac 1 {Z_\beta} (1-\beta') \int_{G} (x^{i+1}-x^i)^{\beta'-2}  \prod_{\substack{j=0\\ j\not=i}}^N (x^{j+1}-x^j)^{\beta'-1} dx \\
\nu_2(G)&=\frac 1 {Z_\beta} (1-\beta') \int_{G} (x^{i}-x^{i-1})^{\beta'-2}  \prod_{\substack{j=0\\ j\not=i-1}}^N (x^{j+1}-x^j)^{\beta'-1} dx.
\end{align*}

Set $\partial \Sigma_N^j:=\{ x\in \partial \Sigma_N: \, x^j=x^{j+1} \}$, $j=0,\ldots, N$, and let for some $x_0 \in \partial \Sigma^i$ and $r>0$, $A:=x_0 + [-r,r]^N \cap \overline{\Sigma}_N$ be such that $\dist(A,\partial \Sigma_N^j)>0$ for all $j\not=i$. Furthermore, let $(A_n)_n$ be a sequence of compact subsets of $A$ such that $A_n\uparrow A$ and $\dist(A_n,\partial \Sigma_N^i)>0$ for every $n$.

By the inner regularity of the Radon measures $\nu_1$ and $\nu_2$ we have $v_1(A)=\lim_n \nu_1(A_n)$ and $v_2(A)=\lim_n \nu_2(A_n)$. Since $\beta'-2<-1$, we get by the choice of $A$ that $\nu_1(A)=\infty$, while $\nu_2(A)<\infty$, which contradicts the local finiteness of $\nu$ and $\nu_1$, respectively.
\end{proof}

\def\cprime{$'$}

\end{document}